\documentclass[a4paper,final]{amsart}
\usepackage{amsmath, amssymb}
\usepackage{type1cm}
\usepackage{enumerate}
\usepackage{graphicx}
\usepackage[czech,english]{babel}
\usepackage{mdwtab}
\usepackage[OT1]{fontenc}
\usepackage[all]{xy}
\newcommand{\Ex}[1][C]{\mathrm{Exp}_{\mathcal{#1}}}
\newcommand{\bbbn}{\mathbb{N}}

\newcommand{\bbbr}{\mathbb{R}}

 \newcommand{\rdens}[1]{\nabla_{{#1}}}

\newcommand{\shm}{\,{\triangledown}\,}
\theoremstyle{plain}
\newtheorem{theorem}{Theorem}

\newtheorem{lemma}{Lemma}
\newtheorem{corollary}{Corollary}
\theoremstyle{definition}

\newtheorem{definition}{Definition}

\theoremstyle{remark}

\newtheorem{problem}{Problem}

\title[circular chromatic number of large girth graphs]{A note on circular chromatic number of graphs with large girth and similar problems}
\thanks{Supported by grant ERCCZ LL-1201 
and CE-ITI,  by the European Associated Laboratory ``Structures in
Combinatorics'' (LEA STRUCO) P202/12/6061, 
and partially supported by ANR project Stint under reference ANR-13-BS02-0007}
\author{Jaroslav Ne{\v s}et{\v r}il}
\address{Jaroslav Ne{\v s}et{\v r}il\\
Computer Science Institute of Charles University (IUUK and ITI)\\
   Malostransk\' e n\' am.25, 11800 Praha 1, Czech Republic}
\email{nesetril@kam.ms.mff.cuni.cz}
\author{Patrice~Ossona~de~Mendez}
\address{Patrice~Ossona~de~Mendez\\
Centre d'Analyse et de Math\'ematiques Sociales (CNRS, UMR 8557)\\
  190-198 avenue de France, 75013 Paris, France
  and
     Computer Science Institute of Charles University (IUUK)\\
   Malostransk\' e n\' am.25, 11800 Praha 1, Czech Republic}
 \email{pom@ehess.fr}
 \date{\today}
 
\begin{document}
 \begin{abstract}
 In this short note, we extend the result of Galluccio, Goddyn, and Hell, which states
that graphs of large girth excluding a minor are nearly bipartite.
We also prove a similar result for the oriented chromatic number, from which follows in particular
that graphs of large girth excluding a minor have oriented chromatic number at most $5$, and for the
$p$th chromatic number $\chi_p$, from which follows in particular
that graphs $G$ of large girth excluding a minor have $\chi_p(G)\leq p+2$.
 \end{abstract}
 \maketitle
 \section{Introduction}
The circular chromatic number $\chi_c(G)$ of a graph $G$, which is
a refinement of its chromatic number, has received much attention recently
(see~\cite{raey} for a survey on recent developments).

Recall that the {\em circular chromatic number}
$\chi _{c}(G)$ of a graph $G$ is the infimum of rational numbers ${\frac  {n}{k}}$ 
such that there is a mapping from  the vertex set of $G$ to ${\mathbb  Z}_n$
with the property that for adjacent vertices are mapped to elements at distance $\geq k$.

 In general, graphs of large girth can have arbitrary given 
 circular chromatic number:
 \begin{theorem}[Ne\v{s}et\v{r}il and Zhu~\cite{Nevsetvril2001}]
 For any rational $r\geq 2$ and any positive integers $t,l$, there is a graph
 $G$ of girth at least $l$ such that $G$ has exactly $t$ $r$-colorings, up to
 equivalence. In particular, for any $r\geq 2$ and for any integer $l$, there is 
 a graph $G$ of girth at least $l$ which is uniquely $r$-colorable and hence it
 has $\chi_c(G)=r$.
\end{theorem}
 
 However, if restricted to special classes of graphs, large girth graphs may be 
 forced to have small circular chromatic number. For instance:
 
  \begin{theorem}[Galluccio, Goddyn, and Hell~\cite{Galluccio2001}]
  \label{thm:ggh}
  For any integer $n\geq 4$, for any $\epsilon>0$, there is an integer $g$ such that every 
  $K_n$-minor free graph $G$ with girth at least $g$ has
 $\chi_c(G)\leq 2+\epsilon$.
 \end{theorem}

The aim of this paper is to extend Theorem~\ref{thm:ggh} to other classes, all with a bounded
expansion.
Classes with bounded expansion have been introduced in
\cite{Taxi_stoc06,POMNI}, and are based on the requirement for the graph
invariant $\rdens{r}(G)$ to be bounded in the class for each $r$.

Denote by $V(G)$ and $E(G)$ the vertex set and the edge set of
$G$. Also denote by $|G|=|V(G)|$ (resp. $\|G\|=|E(G)|$) the {\em
order} of $G$ (resp.  {\em size}).
Let $G,H$ be graphs with $V(H)=\{v_1,\dots,v_h\}$ and let $r$ be an integer. 
A graph $H$ is a {\em shallow minor} of a graph $G$ {\em at depth} $r$, if
there exists disjoint subsets $A_1,\dots,A_h$ of $V(G)$ such that
\begin{itemize}
  \item the subgraph of $G$ induced by $A_i$ is connected and as radius at most
  $r$,
  \item if $v_i$ is adjacent to $v_j$ in $H$, then some vertex in $A_i$ is
  adjacent in $G$ to some vertex in $A_j$.
\end{itemize}
We denote \cite{POMNI, Sparsity} by $G\shm r$  the class of the (simple) graphs
which are shallow minors of $G$ at depth $r$,
and we denote by $\rdens{r}(G)$ the maximum density of a graph in
$G\shm r$, that is:
$$\rdens{r}(G)=\max_{H\in G\shm r}\frac{\|H\|}{|H|}$$
The {\em expansion} of a class $\mathcal C$ is the function $\Ex:\bbbn\rightarrow\bbbr\cup\{\infty\}$ defined by
$$\Ex(r)=\sup_{G\in\mathcal{G}}\rdens{r}(G).$$
A class $\mathcal C$ has {\em bounded expansion} if $\Ex(r)<\infty$  for each value of $r$.

For instance, the class $\mathcal D$ of all graphs with maximum degree $3$ has $\Ex[D](r)=(3/2)2^r$, while
a class $\mathcal{C}$ has uniformly bounded $\Ex$  (that is: $\Ex(r)\leq C$ for some constant $C$, independently of $r$) if and only if there is a graph $F$ that is
a minor of no graph in  $\mathcal C$. 
The expansion function $\Ex$ is indeed non-decreasing and no other general constraint exists on the grow rate of expansion functions. In particular,
for every non-decreasing function $f:\bbbn\rightarrow\bbbn$ with $f(0)=2$ there exists a class $\mathcal{C}$
with $\Ex(r)=f(r)$ (see \cite{Sparsity}, Exercice 5.1).

However, when studying properties of a class, it is sometimes the case that the frontier between classes that verify the property and those that do not
can be (at least approximately) expressed by means of a threshold expansion function (see for instance~\cite{Dvovrak2009}, where it is proved
that every class $\mathcal{C}$ with  expansion $\Ex(r)\leq c^{r^{1/3-\epsilon}}$ is small, but that some non-small class $\mathcal C$ exists, for which
$\Ex(r)\leq 6\cdot 3^{\sqrt{r \log(r+e)}}$).

In such a setting, Theorem~\ref{thm:ggh} can be restated as the fact that if   $\Ex(r)\leq C$ (for some constant $C$) then
 for any $\epsilon>0$, there is an integer $g$ such that every 
 $G\in\mathcal{C}$ with girth at least $g$ has
 $\chi_c(G)\leq 2+\epsilon$.
However, such a statement cannot be extended to all classes with bounded expansion, and even to all classes $\mathcal{C}$ with exponential 
expansion: there exists $3$-regular graphs with
arbitrary large girth and circular chromatic number at least
$7/3$~\cite{Hatami2005}, thus the class $\mathcal D$ of all graphs with maximum degree $3$ (which is such that $\Ex[D](r)=(3/2)2^r$)
does not have the property that large girth implies a circular chromatic number arbitrarily close to $2$.

We show in this paper (Corollary~\ref{cor:circ}) that exponential expansion is indeed a threshold, in the sense that the condition
$\log\Ex(r)/r=o(1)$ (as $r\rightarrow\infty$) is sufficient to ensure the conclusion of Theorem~\ref{thm:ggh}.
Hence a natural problem arises, to make this threshold more precise:

\begin{problem}
\label{pb:circ}
What is the maximum real integer $c$ such that for every class 
 $\mathcal C$  with
$$\lim_{r\rightarrow\infty}\frac{\log\Ex(r)}{r}\leq c$$ 
and every positive real $\epsilon>0$ there is
an integer $g$ with the property that every  graph 
$G\in\mathcal C$ with girth at least $g$ has
 $\chi_c(G)\leq 2+\epsilon$?
\end{problem}

It follows from \cite{Hatami2005} and Corollary~\ref{cor:circ} that $0\leq c<\log 2$, and we conjecture $c=0$.

Other application of our approach are Theorem~\ref{thm:orient1}, which concerns the oriented chromatic number and which is
the subject of Section~\ref{sec:circ}, and Theorem~\ref{thm:chip},.which concerns the $p$th chromatic number $\chi_p$ and which is
the subject of Section~\ref{sec:chip}.

\section{Path-degeneracy}
One of the main tools used to address homomorphism properties of graphs (or directed graphs) with high-girth 
is the notion of path-degeneracy (see for instance \cite{Galluccio2001,Nevsetvril1997}.
\begin{definition}
A graph $G$ is {\em $p$-path degenerate} if there is a sequence
$G=G_0,G_1,\dots, G_t$ of subgraphs of $G$ such that $G_t$ is
a forest, and each $G_i$ ($i>0$) is obtained from $G_{i-1}$ by deleting the
internal vertices of a path of length at least $p$, all of degree two.
\end{definition}

We shall relate path-degeneracy of large girth graphs in a class $\mathcal C$ to expansion properties
of the class $\mathcal{C}$. First we recall a result that will needed for the proof.

\begin{theorem}[Diestel and Rempel \cite{Diestel2004}]
\label{thm:dr}
Let $d$ be an integer and let $G$ be a graph.
If ${\rm girth}(G)>6d+3$ and $\delta(G)\geq 3$ then 
$\rdens{2d}(G)\geq 2^d$
\end{theorem}

We now explicit the connection between path-degeneracy and expansion properties.

\begin{lemma}
\label{lem:2}
Let $\mathcal C$ be a class of graphs, let $p,g\in\bbbn$, and 
let $r=\lceil pg/3\rceil$.
If it holds
 $$
 \frac{\log\Ex(r)}{r}\leq \frac{1}{5p}.
 $$

Then every graph $G$ in $\mathcal C$ with girth at least $g$
is $p$-path degenerate.
\end{lemma}
\begin{proof}
Our proof follows similar lines as the proof (\cite{Galluccio2001}, Lemma~2.5).
Assume for contradiction that $\mathcal C$ include graphs with girth at
least $g$ that are not $p$-path degenerate, and let $G$ be a minimal such
graph.
 The graph $G$ is $2$-connected, since a graph is $p$-path
degenerate if all its blocks are. 
The graph $G$ is not a circuit since, as $l\geq p+1$, the graph $G$
would be $p$-path degenerate.
Hence $G$ is neither an edge or a circuit.
Therefore, there exists a unique graph $G'$ with minimum
degree at least $3$, which is homeomorphic to $G$. The graph $G$ can be
obtained from $G'$ by replacing each edge $e\in E(G')$ with a path $P(e)$ of
length at least one and at most $(p-1)$. Thus $G'\in\mathcal C\shm (p-1)/2$.
Let $d=(g-3)/6$.
According to Theorem~\ref{thm:dr},
if $G'$ has girth greater than $l$ then $\rdens{2d}(G')\geq 2^d$.
As $(\mathcal C\shm (p-1)/2)\shm
2d\subseteq \mathcal C\shm (2dp+(p-1)/2)$ it holds 
 $$\rdens{\lceil lp/3\rceil}(G)\geq\rdens{2dp+(p-1)/2}(G)\geq \rdens{2d}(G')\geq
 2^d.$$
Hence it holds
$$
\frac{\log\rdens{r}(G)}{r}\geq \frac{1}{p}\cdot\frac{d\log
2}{2d+1}> \frac{1}{5p},$$
what contradicts our assumption.
\end{proof}

\section{Circular Chromatic Number of Graphs with Large Girth}
\label{sec:circ}

We shall make use of the following property.

 \begin{lemma}[Bondy and Hell \cite{Bondy1990}]
For every graph $G$ and integer $k$, the following are equivalent:
\begin{itemize}
  \item $\chi_c(G)\leq 2+\frac{1}{k}$
  \item $G\rightarrow C_{2k+1}$
\end{itemize}
\end{lemma}

Existence of a homomorphism to an odd cycle is also linked to path-degenacy.

\begin{lemma}[Galluccio, Goddyn, and Hell~\cite{Galluccio2001}]
\label{lem:ggh}
A $p$-path degenerate graph $G$ admits a homomorphism to any odd circuit of
length at most $p+1$.

In other words, a $p$-path degenerate graph $G$ has
$$
\chi_c(G)\leq 2+\frac{1}{\lfloor p/2\rfloor}.
$$
\end{lemma}

We deduce the following extension of Theorem~\ref{thm:ggh}.
\begin{theorem}
Let $\mathcal C$ be a class of graphs and let $\epsilon>0$.
If it holds
 $$
\liminf_{r\rightarrow\infty}\frac{\log\Ex(r)}{r}\leq \frac{\epsilon}{10}.
$$
Then there is an integer $g$ such that every graph 
$G\in\mathcal C$ with girth at least $g$ has
 $\chi_c(G)\leq 2+\epsilon$.
\end{theorem}
\begin{proof}
Let $p=\lfloor 2/\epsilon\rfloor$.
As 
$$
\liminf_{r\rightarrow\infty}\frac{\log\Ex(r)}{r}=
\liminf_{r\rightarrow\infty}\sup_{G\in\mathcal
C}\frac{\log\rdens{r}(G)}{r}\leq \frac{\epsilon}{10}\leq \frac{1}{5p}
$$
there exists $r\in\bbbn$ such that
$$
\sup_{G\in\mathcal
C}\frac{\log\rdens{r}(G)}{r}\leq \frac{1}{5p}.
$$
Let $g=[3r/p]$.
According to Lemma~\ref{lem:2}, every graph $G\in\mathcal{C}$ with girth at least $g$ is $p$-degenerate.
Thus, according to Lemma~\ref{lem:ggh}, every graph $G\in\mathcal{C}$ with girth at least $g$ is such that
$\chi_c(G)\leq 2+\epsilon$.
\end{proof}
\begin{corollary}
\label{cor:circ}
Let $\mathcal C$ be a class of graphs such that
 $$
\lim_{r\rightarrow\infty}\frac{\log\Ex(r)}{r}=0.
$$
Then for every positive real $\epsilon>0$ there is an integer $g$ such that every graph 
$G\in\mathcal C$ with girth at least $g$ has
 $\chi_c(G)\leq 2+\epsilon$.
\end{corollary}

\section{Oriented Chromatic Number of Graphs with Large Girth}

Recall that the {\em oriented chromatic number} $\chi_o(G)$
 of a (simple) graph $G$ is the minimum $k$ such that every orientation of $G$ admits a homomorphism into some simple digraph with $k$ vertices.
 
It was proved in~\cite{Nevsetvril1997} that there are planar graphs with arbitrarily large girth having oriented chromatic number $5$, and that every planar graph with girth at least $16$ has oriented chromatic number at most $5$. The bound on the girth was reduced to $13$~\cite{Borodin2004} (where the result generalized to graphs embeddable on the torus, or the Klein bottle)
and then to $12$ ~\cite{Borodin2007}.
In ~\cite{Borodin2004}, it is proved that the considered high-girth graphs not only have oriented chromatic number $5$, but
that every orientation of these have a homomorphism to the same $5$-vertex regular tournament $\vec{C}_5^2$. This tournament is a particular case of
the (circular) digraphs $\vec{C}_n^d$, which are defined as the digraphs  with vertex set $\mathbb{Z}_n$ and arcs
$(i,j)$ when $j-i\in{1,\dots,d}$ (mod $n$). 

The following lemma relates $p$-path degeneracy and existence of homomorphism from a directed graph to
$\vec{C}_n^d$.

\begin{lemma}[Ne\v{s}et\v{r}il, Raspaud, and Sopena \cite{Nevsetvril1997}]
\label{lem:nrs}
Then end-vertices of every oriented path $\vec P$ of length at least $(n - 1)/(d - 1)$ can be mapped by a homomorphism
$\vec P \rightarrow \vec{C}_n^d$ to any pair of (not necessarily distinct) vertices of $\vec{C}_n^d$.
\end{lemma}
We can now relate expansion properties of a class to the oriented chromatic number of large girth graphs in the class.
\begin{theorem}
\label{thm:orient1}
Let $\mathcal C$ be a class of graphs, and let $n,d\in\bbbn$.
Assume it holds
 $$
 \frac{\log\Ex(r)}{r}\leq \frac{d-1}{5(n-1)}.
 $$

Then there is an integer $g$ such that for every orientation $\vec{G}$ of a graph $G$ in $\mathcal C$ with girth at least $g$ there exists a homomorphism
$\vec{G}\rightarrow \vec{C}_n^d$.
\end{theorem}
\begin{proof}
Under the assumptions of the theorem there exists, according to Lemma~\ref{lem:2} an integer $g$ such that
every graph $G\in\mathcal{C}$ with girth at least $g$ is $(n-1)/(d-1)$-path degenerate. For every orientation $\vec{G}$ of
$G$, the existence of a homomorphism $f:\vec{G}\rightarrow\vec{C}_n^d$ the easily follows by an iterative use
of Lemma~\ref{lem:nrs}.
\end{proof}
From this theorem immediately follows:
\begin{theorem}
Let $\mathcal C$ be a class of graphs such that
 $$
\limsup_{r\rightarrow\infty}\frac{\log\Ex(r)}{r}\leq \frac{1}{20}
$$
Then there is an integer $g$ such that every 
$G\in\mathcal C$ with girth at least $g$ has
oriented chromatic number at most $5$ (and actually are homomorphic to the tournament $\vec{C}_5^2$).
\end{theorem}
\begin{proof}
This follows from Theorem~\ref{thm:orient1}, by considering $d=2$ and $n=5$.
\end{proof}
In particular, we have
\begin{corollary}
Let $\mathcal C$ be a proper minor closed class of graphs.
Then there is an integer $g$ such that every 
$G\in\mathcal C$ with girth at least $g$ 
has $\chi_o(G)\leq 5$.
\end{corollary}
\section{Centered Coloring of Graphs with Large Girth}
\label{sec:chip}
Recall that the {\em tree-depth} ${\rm td}(G)$ of a graph $G$ is the minimum height of a rooted forest $F$ such that
$G$ is the subgraph of the closure of $F$.
For positive integer $p$, the {\em $p$th chromatic number} $\chi_p(G)$ is the least number of colors in a vertex coloring of $G$ such that every $i\leq p$ colors induce a subgraph with tree-depth at most $i$.
A {\em $p$-centered coloring} of a graph $G$ is a vertex coloring such that, for any (induced) connected subgraph $H$ , either some color $c(H)$ appears exactly once in $H$, or $H$ gets at least $p$ colors.
For detailed properties of tree-depth and $\chi_p$ we refer the reader to \cite{Sparsity}.
The $\chi_p$ invariants  are related to $p$-centered coloring as follows.
\begin{lemma}[Ne{\v s}et{\v r}il and Ossona de Mendez \cite{Taxi_tdepth}]
For every graph $G$ and every integer $p$, $\chi_p(G)$ is the minimum number of colors in a $(p+1)$-centered coloring of $G$.
\end{lemma}

\begin{lemma}
\label{lem:chip}
Let $G$ be a graph and let $p$ be a positive integer. If $G$ is $2p$-path degenerate, then $G$ has a $p$-centered coloring
with at most $p+1$ colors.
\end{lemma}
\begin{proof}
According to the definition of  $p$-degeneracy,  there is a sequence
$$G=G_0\supset G_1\supset\dots\supset G_t$$
 of subgraphs of $G$ such that $G_t$ is
a forest, and each $G_i$ ($i>0$) is obtained from $G_{i-1}$ by deleting the
internal vertices of a path $P_i$ of length at least $p$, all of degree two.
For each $0<i\leq t$ select an internal vertex $v_i$ of $P_i$ at distance at least $p$ 
(in $P_i$) from both endvertices of $P_i$. and orient all the edges of $P_i$ toward $v_i$.
Let $C=\{v_i: 0<i\leq t\}$. Then $G-C$ is a rooted forest, with edges oriented from the roots.
Color the vertices $G-C$ with $p$ colors inductively as follows: each root is colored $0$, and each other vertex $u$
is colored as $(c+1) \bmod p$, where $c$ is the color assigned to the father of $u$. Note that this coloring is
obviously a $p$-centered coloring of $G-C$. Now color the vertices in $C$ with color $p$. It is easily checked that
every path linking two vertices in $C$ already got at least $p-1$ colors on the internal vertices  by the above coloring  
of $G-C$. It follows that the every $(p-1)$-colored connected component of $G$ is either included in $G-C$ (hence
has a uniquely colored vertex), or contains exactly one vertex in $C$, that is exactly one vertex colored $p$. Thus
the defined coloring is a $p$-centered coloring of $G$.
\end{proof}

The following theorem now easily follows.
\begin{theorem}
\label{thm:chip}
Let $\mathcal C$ be a class of graphs, and let $p\in\bbbn$.
Assume it holds
 $$
 \frac{\log\Ex(r)}{r}\leq \frac{1}{10(p+1)}.
 $$

Then there is an integer $g$ such that every  graph $G$ in $\mathcal C$ with girth at least $g$ has $\chi_p(G)\leq p+2$.
\end{theorem}
\begin{proof}
According to Lemma~\ref{lem:2} there is an integer $g$ such that
every graph $G\in\mathcal{C}$ with girth at least $g$ is $(2p+2)$-path degenerate thus,
 according to Lemma~\ref{lem:chip}, such that $\chi_p(G)\leq p+2$.
\end{proof}
In particular, we have
\begin{corollary}
Let $\mathcal C$ be a proper minor closed class of graphs.
Then for every integer $p$ there is an integer $g$ such that every 
$G\in\mathcal C$ with girth at least $g$ 
has $\chi_p(G)\leq p+2$.
\end{corollary}

As $\chi_p(P_n)=p+1$ for sufficiently large $n$, it follows that the upper bound cannot be lowered under $p+1$.
The bounds $\chi_1(G)\leq 3$ and $\chi_2(G)\leq 4$ are
optimal, as witnessed by a long odd cycle and the series-parallel graphs depicted below, respectively.
\vspace{5mm}
\begin{center}
\includegraphics[scale=.4]{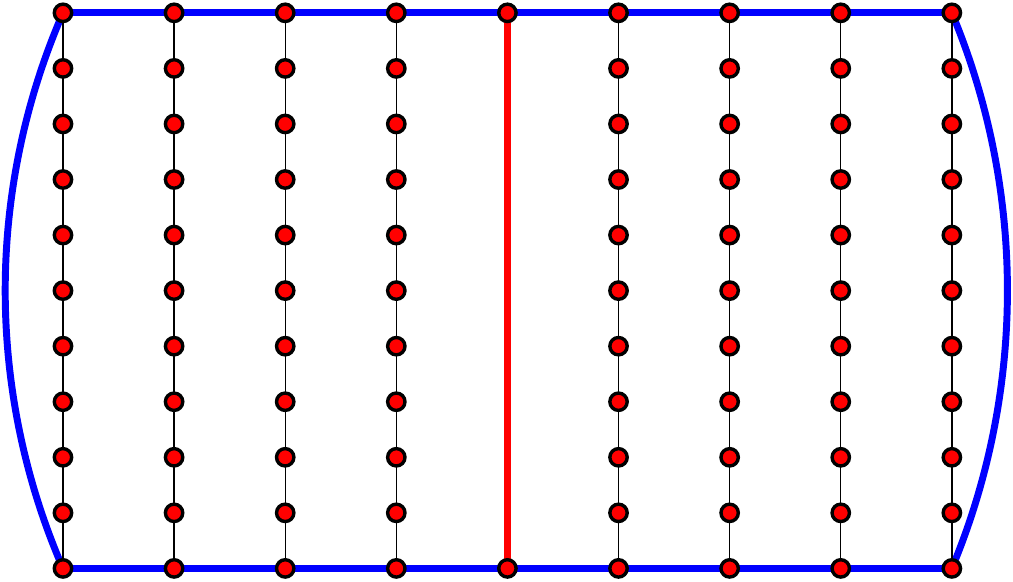}
\end{center}
We leave as a problem the question whether the bound $p+2$ is tight for every integer $p$. 
\section{Concluding Remarks}
Remark that the property of a class that high girth graphs in the class are nearly bipartite or that they have oriented chromatic number at most $5$
is not ``topological'' in the sense that satisfaction (or non-satisfaction) of this property is not preserved when the graphs in the class are subdivided:
Consider an arbitrary class of graphs $\mathcal C$.  Construct the class $\mathcal C'$ as follows: for each graph $G\in\mathcal C$ (with $|G|$ vertices)
put in $\mathcal C'$ the $|G|$-subdivision of $G$. Then the class $\mathcal C'$ has the property that high girth graphs in the class are 
nearly bipartite and have oriented chromatic number at most $5$, although this does not have to hold for $\mathcal C$. In this sense, considering
classes defined by a forbidden minor does not seem to be here optimal.

It is possible, from a class $\mathcal C$, to construct the sequence (which we call {\em resolution} of $\mathcal C$) of the
classes of shallow minors of graphs in $\mathcal C$ at increasing depth
$$\mathcal C\subseteq\{G\shm 0: G\in\mathcal C\}\subseteq\dots\subseteq \{G\shm r: G\in\mathcal C\}\subseteq\dots$$
and we note that the results we obtained in this paper are expressed by means of the growth rate of the average degrees of the graphs
in these classes, as $r\rightarrow\infty$.

A related approach consists in considering the growth rate of the clique number of the graphs
in these classes, as $r\rightarrow\infty$. For a graph $G$ and an integer $r$ define
$$\omega_r(G)=\max\{\omega(H): H\in G\shm r\}.$$
Then we have the following connection between the grow rate of $\omega_r$ and the existence of
sublinear vertex separator in a class of graphs: 
Then the following holds
\begin{theorem}[\cite{Taxi_stoc06, Sparsity}]
\label{thm:sep}
Let $\mathcal{C}$ be a class of graphs such that
$$\lim_{r\rightarrow\infty}\frac{\log\sup_{G\in\mathcal C}\omega_r(G)}{r}=0.$$
Then the graphs of order $n$ in $\mathcal C$ have vertex separators of size $s(n) = o(n)$ which may be computed in time $O(ns(n)) = o(n^2)$.
\end{theorem}

Note that this threshold is sharp, in the sense that for every $c>0$ there is a class 
 $\mathcal C$  with
$$\lim_{r\rightarrow\infty}\frac{\log \sup_{G\in\mathcal C}\omega_r(G)}{r}\leq c$$ 
but no sublinear vertex separators 
(consider the class $\mathcal C$ of $k$-subdivisions of cubic graphs).

As $\omega_r(G)\leq 2\rdens{r}(G)+1$,
 the classes $\mathcal C$ with $\lim_{r\rightarrow\infty}(\log\Ex(r))/r=0$, which we considered in this note, satisfy the conditions of Theorem~\ref{thm:sep}.
 However, the condition on $\omega_r$ is weaker, as witnessed by classes of d-dimensional meshes with bounded aspect ratio (see~\cite{Miller1990} for a definition),
 that are such that $\sup_{G\in\mathcal C}\omega_r(G)$ grows polynomially with $r$ (see \cite{shallow}).
 
 We believe that the study of the growth rate of graph parameters on the resolution of a class can be useful to determine threshold where certain class properties stop 
 being true.
\providecommand{\noopsort}[1]{}\providecommand{\noopsort}[1]{}
\providecommand{\bysame}{\leavevmode\hbox to3em{\hrulefill}\thinspace}
\providecommand{\MR}{\relax\ifhmode\unskip\space\fi MR }
\providecommand{\MRhref}[2]{%
  \href{http://www.ams.org/mathscinet-getitem?mr=#1}{#2}
}
\providecommand{\href}[2]{#2}

 \end{document}